\documentclass[11pt]{amsart}
\usepackage[margin=30mm]{geometry}
\usepackage{amsmath,amssymb}
\usepackage{amsthm}
\usepackage{mathrsfs}

\newtheorem{thm}{Theorem}[section]

\newtheorem{lem}{Lemma}[section]
\newtheorem{cor}{Corollary}[section]
\newtheorem{defi}{Definition}[section]
\newtheorem{ex}{Example}[section]
\newtheorem{rem}{Remark}[section]

\newtheorem*{notation}{\it{Notation}}

\usepackage{enumitem}

\begin{document}

\title{Some consequences of shadowing and entropy-expansiveness}
\author{Noriaki Kawaguchi}
\address{Department of Mathematical and Computing Science, School of Computing, Institute of Science Tokyo, 2-12-1 Ookayama, Meguro-ku, Tokyo 152-8552, Japan}
\email{gknoriaki@gmail.com}

\begin{abstract}
Given any continuous self-map of a compact metric space, we prove that a non-empty uniformly rigid subset must be zero-dimensional, under assumptions of shadowing and h-expansiveness.
\end{abstract}

\maketitle

\markboth{NORIAKI KAWAGUCHI}{Some consequences of shadowing and entropy-expansiveness}

\section{Introduction}

{\em Shadowing} is an important concept in the theory of dynamical systems. It describes the phenomenon in which coarse orbits, or {\em pseudo-orbits}, are closely followed by true orbits. The {\em shadowing lemma} states that a $C^1$-diffeomorphism of a closed differentiable manifold exhibits shadowing on a neighborhood of its hyperbolic set. The implications of shadowing have been studied in various studies, including \cite{LO1,LO2,Moo,MO}, often in topological settings. For an overview of shadowing theory, we refer the reader to the monographs \cite{AH,P}.

The notion of {\em rigidity}, introduced in \cite{GM} as uniform recurrence in topological dynamics, has attracted considerable attention. In particular, \cite{AGHSY} provides results on rigid subsets and their relation to chaotic properties in dynamical systems. A result of \cite{AGHSY} is that if a continuous self-map of an infinite compact metric space is transitive, then there is an increasing sequence of uniformly rigid Cantor sets whose union is dense in the phase space.

In this paper, we focus on the dimension of uniformly rigid subsets. Under assumptions of shadowing and h-expansiveness, we show that a non-empty uniformly rigid subset must be zero-dimensional. A related result of \cite{Me} is that if a continuous self-map of a compact metric space is uniformly rigid and has shadowing, then the entire space is zero-dimensional. In contrast to \cite{Me}, the present work deals with uniformly rigid subsets that are not necessarily invariant, as well as closed invariant uniformly rigid subsets equipped with (not necessarily intrinsic) shadowing.

We begin by defining shadowing on a subset. Throughout, $X$ denotes a compact metric space endowed with a metric
\[
d\colon X\times X\to[0,\infty).
\]

\begin{defi}
\normalfont
Let $f\colon X\to X$ be a continuous map and let $\xi=(x_i)_{i\ge0}$ be a sequence of points in $X$.
\begin{itemize}
\item For $\delta>0$, $\xi$ is called a {\em $\delta$-pseudo orbit} of $f$ if
\[
\sup_{i\ge0}d(f(x_i),x_{i+1})\le\delta.
\]
\item For $x\in X$ and $\epsilon>0$, $\xi$ is said to be {\em $\epsilon$-shadowed} by $x$ if
\[
\sup_{i\ge0}d(x_i,f^i(x))\le\epsilon.
\]
\end{itemize}
Given a closed subset $Y$ of $X$, we say that $f$ has {\em shadowing} on $Y$ if for any $\epsilon>0$, there is $\delta>0$ such that every $\delta$-pseudo orbit $(x_i)_{i\ge0}$ of $f$ with $\{x_i\colon i\ge0\}\subset Y$ is $\epsilon$-shadowed by some $x\in X$.
\end{defi}

\begin{rem}
\normalfont
Let $\delta,\epsilon>0$ be as in the definition of shadowing on $Y$. Then, for any sequence $(x_i)_{i\in\mathbb{Z}}$ of points in $Y$ with
\[
\sup_{i\in\mathbb{Z}}d(f(x_i),x_{i+1})\le\delta,
\]
there are $x\in X$ and a sequence $(y_i)_{i\in\mathbb{Z}}$ of points in $X$ such that $y_0=x$, $f(y_i)=y_{i+1}$ for all $i\in\mathbb{Z}$, and
\[
\sup_{i\in\mathbb{Z}}d(x_i,y_i)\le\epsilon.
\]
\end{rem}

Next, we recall the definition of h-expansiveness \cite{B1}. Let $f\colon X\to X$ be a continuous map and let $K$ be a subset of $X$. For $n\ge1$ and $r>0$, a subset $E$ of $K$ is said to be {\em $(n,r)$-separated} if 
\[
\max_{0\le i\le n-1}d(f^i(x),f^i(y))>r
\]
for all $x,y\in E$ with $x\ne y$. Let $s_n(f,K,r)$ denote the largest cardinality of an $(n,r)$-separated subset of $K$. We define $h(f,K,r)$ and $h(f,K)$ by
\[
h(f,K,r)=\limsup_{n\to\infty}\frac{1}{n}\log{s_n(f,K,r)}
\]
and $h(f,K)=\lim_{r\to0}h(f,K,r)$. The topological entropy $h_{\rm top}(f)$ of $f$ is defined by $h_{\rm top}(f)=h(f,X)$ (see, e.g., \cite{W} for more details).

\begin{defi}
\normalfont
For a continuous map $f\colon X\to X$, let
\[
\Phi_{\epsilon}(x)=\{y\in X\colon\sup_{i\ge0}d(f^i(x),f^i(y))\le\epsilon\}
\]
for $x\in X$ and $\epsilon>0$; and let
\[
h_f^\ast(\epsilon)=\sup_{x\in X}h(f,\Phi_{\epsilon}(x))
\]
for $\epsilon>0$. We say that $f$ is {\em h-expansive} if $h_f^\ast(\epsilon)=0$ for some $\epsilon>0$. 
\end{defi}

\begin{rem}
\normalfont
If a continuous map $f\colon X\to X$ is h-expansive, then $h_{\rm top}(f)<\infty$.
\end{rem}

The definition of zero-dimensional subsets is as follows. For background on the dimension theory, we refer the reader to \cite{Mi}. 

\begin{defi}
\normalfont
We say that a subset $S$ of $X$ is {\em zero-dimensional} if it is not empty and if for any $x\in S$ and any open subset $U$ of $S$ with $x\in U$, there is a clopen subset $V$ of $S$ such that $x\in V\subset U$.
\end{defi}

\begin{rem}
\normalfont
A subset $S$ of $X$ is said to be {\em totally disconnected} if it is empty or if every connected component of $S$ is a singleton. We know that if a subset $S$ of $X$ is zero-dimensional, then $S$ is totally disconnected; and the converse holds when $S$ is a non-empty closed subset of $X$.
\end{rem}

Let us recall the definition of uniformly rigid subsets \cite{AGHSY}.

\begin{defi}
\normalfont
Given a continuous map $f\colon X\to X$, we say that a subset $S$ of $X$ is {\em uniformly rigid} if for any $\gamma>0$, there is $l\ge1$ such that
\[
\sup_{x\in S}d(x,f^l(x))\le\gamma.
\]  
\end{defi}

\begin{rem}
\normalfont
Let $f\colon X\to X$ be a continuous map. If a subset $S$ of $X$ is uniformly rigid, then
\begin{itemize}
\item every subset of $S$ is uniformly rigid,
\item the closure $\overline{S}$ is uniformly rigid.
\end{itemize}
\end{rem}

Let $f\colon X\to X$ be a continuous map and let $Y$ be a closed subset of $X$. For $x\in X$, we denote by $d(x,Y)$ the distance of $x$ from $Y$:
\[
d(x,Y)=\inf_{y\in Y}d(x,y).
\]
For $r>0$, we define $Y_r^f$ as the set of $x\in X$ for which there is a sequence $(y_i)_{i\in\mathbb{Z}}$ of points in $X$ such that $y_0=x$, $f(y_i)=y_{i+1}$ for all $i\in\mathbb{Z}$, and
\[
\sup_{i\in\mathbb{Z}}d(y_i,Y)\le r.
\]
Note that $Y_r^f$ is a closed subset of $X$ and satisfies $f(Y_r^f)=Y_r^f$.

\begin{rem}
\normalfont
If $f\colon X\to X$ is a homeomorphism, then
\[
Y_r^f=\{x\in X\colon\sup_{i\in\mathbb{Z}}d(f^i(x),Y)\le r\}
\]
for all $r>0$.
\end{rem}

The main result of this paper is the following theorem.

\begin{thm}
Given a continuous map $f\colon X\to X$ and a closed subset $Y$ of $X$, if
\begin{itemize}
\item $f(Y)\subset Y$,
\item $f$ has shadowing on $Y$,
\item $f|_{Y_b^f}$ is h-expansive for some $b>0$,
\end{itemize}
then every non-empty uniformly rigid subset $S$ of $Y$ is zero-dimensional.
\end{thm}

The following two examples show that the assumptions of shadowing and h-expansiveness in Theorem 1.1 are both necessary. Given a continuous map $f\colon X\to X$, we denote by $Fix(f)$ the fixed points for $f$:
\[
Fix(f)=\{x\in X\colon f(x)=x\}.
\]
\begin{ex}
\normalfont
Let $I=[0,1]$ and let $f=id_I\colon I\to I$, the identity map. By $h_{\rm top}(f)=0$, $f$ is h-expansive. We easily see that $f$ does not have shadowing on $I$. Note that $I(=Fix(f))$ is uniformly rigid but not zero-dimensional. 
\end{ex}

\begin{ex}
\normalfont
Let $X=[0,1]^\mathbb{Z}$ and let $f\colon X\to X$ be the shift map, i.e.,
\[
f(x)_j=x_{j+1}
\]
for all $x=(x_j)_{j\in\mathbb{Z}}\in X$ and $j\in\mathbb{Z}$. We know that $f$ has shadowing on $X$ (see Theorem 2.3.12 of \cite{AH}). Since $h_{\rm top}(f)=\infty$, $f$ is not h-expansive. Note that
\[
Fix(f)=\{(x_j)_{j\in\mathbb{Z}}\in X\colon x_j=x_{j+1}\:\text{for all $j\in\mathbb{Z}$}\}.
\]
Let $S^1=\{z\in\mathbb{C}\colon|z|=1\}$. Let $\alpha\in(0,1)\setminus\mathbb{Q}$ and let $R_\alpha\colon S^1\to S^1$ be the map defined by $R_\alpha(z)=z\cdot e^{2\pi i\alpha}$ for all $z\in S^1$. We know that there is an injective continuous map $h\colon S^1\to X$ such that $h\circ R_\alpha=f\circ h$ (see, e.g., \cite{L}). Letting $Y\in\{Fix(f),h(S^1)\}$, we see that
\begin{itemize}
\item $f(Y)=Y$,
\item $f$ has shadowing on $Y$,
\item $h_{\rm top}(f|_Y)=0$ and thus $f|_Y$ is h-expansive,
\item $Y$ is uniformly rigid;
\end{itemize}
however, $Y$ is not zero-dimensional.
\end{ex}

A homeomorphism $f\colon X\to X$ is said to be {\em expansive} if there is $e>0$ such that
\[
\sup_{i\in\mathbb{Z}}d(f^i(x),f^i(y))\le e
\]
implies $x=y$ for all $x,y\in X$. Let $Y$ be a subset of $X$ and $b>0$. If $f|_{Y_b^f}$ is expansive, then by \cite[Corollary 2.3]{B1}, $f|_{Y_b^f}$ is h-expansive.

By a combination of Proposition 6.4.6, Corollary 6.4.10, and Theorem 18.1.2 of \cite{KH}, we have the following lemma. 

\begin{lem}
Let $M$ be a closed Riemannian manifold and let
\[
f\colon M\to M
\]
be a $C^1$-diffeomorphism. Given any hyperbolic set $\Lambda$ for $f$,
\begin{itemize}
\item $f$ has shadowing on $\Lambda$,
\item $f|_{\Lambda_b^f}$ is expansive for some $b>0$.
\end{itemize}
\end{lem}

From Theorem 1.1 and Lemma 1.1, we obtain the following corollary.

\begin{cor}
Let $M$ be a closed Riemannian manifold and let $f\colon M\to M$ be a $C^1$-diffeomorphism. Let $\Lambda$ be a hyperbolic set for $f$. Then, every non-empty uniformly rigid subset $S$ of $\Lambda$ is zero-dimensional.
\end{cor}

The following lemma is a special case of the so-called {\em countable closed sum theorem} (see Theorem 3.2.8 of \cite{Mi}).

\begin{lem}
For any sequence $A_j$, $j\ge1$, of closed subsets of $X$, if $A_j$ is empty or zero-dimensional for all $j\ge1$, then $\bigcup_{j\ge1}A_j$ is either empty or zero-dimensional.
\end{lem}

Given a continuous map $f\colon X\to X$, we denote by $Per(f)$ the set of periodic points for $f$:
\[
Per(f)=\bigcup_{j\ge1}\{x\in X\colon f^j(x)=x\}.
\]
By Theorem 1.1 and Lemma 1.2, we obtain the following theorem.

\begin{thm}
If a continuous map $f\colon X\to X$ is h-expansive and has shadowing on $X$, then $Per(f)$ is either empty or zero-dimensional.
\end{thm}

\begin{proof}[Proof of Theorem 1.2]
Note that
\[
Fix(f^j)=\{x\in X\colon f^j(x)=x\}
\]
is uniformly rigid for all $j\ge1$. Since $f$ is h-expansive and has shadowing on $X$, by Theorem 1.1, $Fix(f^j)$ is empty or zero-dimensional for all $j\ge1$. Since
\[
Per(f)=\bigcup_{j\ge1}Fix(f^j),
\]
and $Fix(f^j)$ is a closed subset of $X$ for all $j\ge1$, by Lemma 1.2, we conclude that $Per(f)$ is either empty or zero-dimensional.
\end{proof}

We present three examples related to Theorem 1.2.

\begin{ex}
\normalfont
If a homeomorphism $f\colon X\to X$ is expansive, then
\[
Fix(f^j)=\{x\in X\colon f^j(x)=x\}
\]
is a finite set for all $j\ge1$; therefore, $Per(f)$ is a countable set and so zero-dimensional.
\end{ex}

\begin{ex}
\normalfont
Let $C$ be the Cantor ternary set in the interval $[0,1]$ and let $f=id_C\colon C\to C$, the identity map. By $h_{\rm top}(f)=0$, $f$ is h-expansive. We easily see that $f$ has shadowing on $C$. Note that
\[
Per(f)=Fix(f)=C
\]
is zero-dimensional but an uncountable set.
\end{ex}

\begin{ex}
\normalfont
Given a sequence $m=(m_j)_{j\ge1}$ of integers such that $2\le m_1<m_2<\cdots$ and $m_j|m_{j+1}$ for each $j\ge1$, let
\begin{itemize}
\item $X(j)=\{0,1,\dots,m_j-1\}$ (with the discrete topology),
\item
\[
X_m=\{(x_j)_{j\ge1}\in\prod_{j\ge1}X(j)\colon x_j\equiv x_{j+1}\pmod{m_j}\:\:\text{for all $j\ge1$}\},
\]
\item $g_m(x)_j=x_j+1\pmod{m_j}$ for all $x=(x_j)_{j\ge1}\in X_m$ and $j\ge1$.
\end{itemize}
We regard $X_m$ as a subspace of the product space $\prod_{j\ge1}X(j)$. The homeomorphism
\[
g_m\colon X_m\to X_m
\]
is called an {\em odometer} with the periodic structure $m$. Since $g_m$ is {\em equicontinuous} and $X_m$ is totally disconnected, $g_m$ has shadowing on $X_m$ (see, e.g., Theorem 4 of \cite{Moo}). By $h_{\rm top}(g_m)=0$, $g_m$ is h-expansive. Note that $Per(g_m)=\emptyset$. 
\end{ex}

This paper consists of three sections. In Section 2, we prove Theorem 1.1. In Section 3, we consider closed invariant pairwise recurrent subsets and prove a corollary of Lemma 2.3. 

\section{Proof of Theorem 1.1}

In this section, we prove Theorem 1.1. For the proof, we need a few lemmas. For a subset $S$ of $X$, we denote by ${\rm diam}\:S$ the diameter of $S$:
\[
{\rm diam}\:S=\sup_{x,y\in S}d(x,y).
\]
Given a continuous map $f\colon X\to X$ and $\delta>0$, a finite sequence $(x_i)_{i=0}^k$ of points in $X$, where $k\ge1$, is called a {\em $\delta$-chain} of $f$ if
\[
\sup_{0\le i\le k-1}d(f(x_i),x_{i+1})\le\delta.
\]

Our proof of Lemma 2.1 will be based on the argument in the following remark. A similar argument can be found in \cite{Pa} (see \cite[Theorem 11.12]{Pa}).

\begin{rem}
\normalfont
Let $f\colon X\to X$ be a continuous map and let $(a_i)_{i=0}^{l-1}$, $(b_i)_{i=0}^{l-1}$, where $l\ge1$, be sequences of points in $X$. Let $\mathbb{N}_0=\{0\}\cup\mathbb{N}=\{0,1,2,\dots\}$. For $u=(u_j)_{j\ge0}\in\{a,b\}^{\mathbb{N}_0}$, let
\[
w_{jl+r}^u=(u_j)_r
\]
for all $j\ge0$ and $0\le r\le l-1$. Given $\epsilon>0$, letting
\[
W=\{p\in X\colon\text{$(w_i^u)_{i\ge0}$ is $\epsilon$-shadowed by $p$ for some $u\in\{a,b\}^{\mathbb{N}_0}$}\},
\]
we see that $W$ is a closed subset of $X$ with $f^l(W)\subset W$. Let $\Xi$ be a closed subset of $W$ with $f^l(\Xi)\subset\Xi$. If
\[
\sup_{0\le i\le l-1}d(a_i,b_i)>2\epsilon,
\]
then we can define a map $\pi\colon\Xi\to\{a,b\}^{\mathbb{N}_0}$ such that $(w_i^{\pi(p)})_{i\ge0}$ is $\epsilon$-shadowed by $p$ for all $p\in\Xi$. Note that $\pi$ is continuous and satisfies
\[
\pi\circ f^l|_\Xi=\sigma\circ\pi,
\]
where $\sigma\colon\{a,b\}^{\mathbb{N}_0}\to\{a,b\}^{\mathbb{N}_0}$ is the shift map. If, moreover,
\begin{itemize}
\item for every $u\in\{a,b\}^{\mathbb{N}_0}$, $(w_i^u)_{i\ge0}$ is $\epsilon$-shadowed by some $p\in\Xi$,
\end{itemize}
then $\pi$ is a surjective map.
\end{rem}

\begin{lem}
Let $f\colon X\to X$ be a continuous map and let $Y$ be a closed subset of $X$. Given $\Delta,\alpha,\rho>0$, if
\begin{itemize}
\item every $\Delta$-pseudo orbit $(w_i)_{i\ge0}$ of $f$ with $\{w_i\colon i\ge0\}\subset Y$ is $\alpha$-shadowed by some $p\in X$,
\item there are $\Delta$-chains $(a_i)_{i=0}^l$, $(b_i)_{i=0}^l$ of $f$ such that
\begin{itemize}
\item $a_0=b_0=a_l=b_l$,
\item $2\alpha<\sup_{0\le i\le l-1}d(a_i,b_i)\le\rho$,
\item $\{a_i\colon 0\le i\le l-1\}\cup\{b_i\colon 0\le i\le l-1\}\subset Y$,
\end{itemize}
\end{itemize}
then $h^\ast_G(\rho+2\alpha)>0$, where $G=f|_{Y_\alpha^f}\colon Y_\alpha^f\to Y_\alpha^f$.
\end{lem}

\begin{proof}
Given $u=(u_j)_{j\ge0}\in\{a,b\}^{\mathbb{N}_0}$, let
\begin{equation*}
\tilde{u}_j=
\begin{cases}
u_j&\text{if $j\ge0$}\\
a&\text{if $j<0$}
\end{cases}
\end{equation*}
and let
\[
w_{jl+r}^u=(\tilde{u}_j)_r
\]
for all $j\in\mathbb{Z}$ and $0\le r\le l-1$. Since $\{w_i^u\colon i\in\mathbb{Z}\}\subset Y$ and
\[
\sup_{i\in\mathbb{Z}}d(f(w_i^u),w_{i+1}^u)\le\Delta,
\]
there are $p_u\in X$ and a sequence $(w_i)_{i\in\mathbb{Z}}$ of points in $X$ such that $w_0=p_u$, $f(w_i)=w_{i+1}$ for all $i\in\mathbb{Z}$, and
\[
\sup_{i\in\mathbb{Z}}d(w_i^u,w_i)\le\alpha.
\]
It follows that $p_u\in Y_\alpha^f$ and $(w_i^u)_{i\ge0}$ is $\alpha$-shadowed by $p_u$. Letting
\[
\Xi=Y_\alpha^f\cap\{p\in X\colon\text{$(w_i^u)_{i\ge0}$ is $\alpha$-shadowed by $p$ for some $u\in\{a,b\}^{\mathbb{N}_0}$}\},
\]
we see that $\Xi$ is a closed subset of $X$ with $f^l(\Xi)\subset\Xi$ and $\Xi\subset Y_\alpha^f$. We define a map $\pi\colon\Xi\to\{a,b\}^{\mathbb{N}_0}$ so that $(w_i^{\pi(p)})_{i\ge0}$ is $\alpha$-shadowed by $p$ for all $p\in\Xi$. It follows that $\pi$ is a surjective continuous map with 
\[
\pi\circ f^l|_\Xi=\sigma\circ\pi,
\]
where $\sigma\colon\{a,b\}^{\mathbb{N}_0}\to\{a,b\}^{\mathbb{N}_0}$ is the shift map. Note that
\[
\sup_{i\ge0}d(f^i(p),f^i(q))\le\rho+2\alpha
\]
for all $p,q\in\Xi$. By taking $p\in\Xi$, we obtain $\Xi\subset Y_\alpha^f\cap\Phi_{\rho+2\alpha}(p)$ and so
\[
h_G^\ast(\rho+2\alpha)\ge h(f,Y_\alpha^f\cap\Phi_{\rho+2\alpha}(p))\ge h(f,\Xi)\ge\frac{1}{l}h(f^l,\Xi)\ge\frac{1}{l}h(\sigma,\{a,b\}^{\mathbb{N}_0})=\frac{1}{l}\log{2}>0,
\]
thus the lemma has been proved.
\end{proof}

By using Lemma 2.1, we prove Lemma 2.2 and Lemma 2.3.

\begin{lem}
Let $f\colon X\to X$ be a continuous map and let $Y$ be a closed subset of $X$ such that
\begin{itemize}
\item $f(Y)\subset Y$,
\item $f$ has shadowing on $Y$,
\item $f|_{Y_b^f}$ is h-expansive for some $b>0$.
\end{itemize}
Given any $\beta>0$, there is $\gamma>0$ such that for any subset $C$ of $Y$, if
\begin{itemize}
\item ${\rm diam}\:C\le\gamma$,
\item $C$ is connected and uniformly rigid,
\end{itemize}
then $\sup_{i\ge0}{\rm diam}\:f^i(C)\le\beta$.
\end{lem}

\begin{proof}
Let $g=f|_{Y_b^f}\colon Y_b^f\to Y_b^f$. In order to obtain a contradiction, we assume the contrary. Then, there is $0<\epsilon\le b$ such that
\begin{itemize}
\item $h_g^\ast(5\epsilon)=0$,
\item for any $0<\gamma<\epsilon$, there is a subset $C$ of $Y$ such that
\begin{itemize}
\item ${\rm diam}\:C\le\gamma$,
\item $C$ is connected and uniformly rigid,
\item ${\rm diam}\:f^k(C)>3\epsilon$ for some $k\ge1$. 
\end{itemize}
\end{itemize}
We take $\delta>0$ such that every $\delta$-pseudo orbit $(w_i)_{i\ge0}$ of $f$ with $\{w_i\colon i\ge0\}\subset Y$ is $\epsilon$-shadowed by some $p\in X$. Let $0<\gamma<\epsilon$, $C$, and $k\ge1$ as above. Since $C$ is uniformly rigid, we have
\[
\sup_{x\in C}d(x,f^l(x))\le\gamma
\]
for some $l\ge k+1$. Since $C$ is connected, there are $x,y\in C$ such that
\[
\sup_{0\le i\le l}d(f^i(x),f^i(y))=3\epsilon.
\]
We fix $z\in C$. Let
\[
(a_i)_{i=0}^l=(z,f(x),f^2(x),\dots,f^{l-1}(x),z)
\]
and
\[
(b_i)_{i=0}^l=(z,f(y),f^2(y),\dots,f^{l-1}(y),z).
\]
Since $\{x,y,z\}\subset C$ and ${\rm diam}\:C\le\gamma$, if $\gamma$ is sufficiently small, then $(a_i)_{i=0}^l$ and $(b_i)_{i=0}^l$ are $\delta$-chains of $f|_Y$. Note that $a_0=b_0=a_l=b_l$ and
\[
2\epsilon<\sup_{0\le i\le l-1}d(a_i,b_i)=3\epsilon.
\]
By Lemma 2.1, we obtain $h_G^\ast(5\epsilon)>0$, where $G=f|_{Y_\epsilon^f}\colon Y_\epsilon^f\to Y_\epsilon^f$. By $\epsilon\le b$, we obtain $h_g^\ast(5\epsilon)\ge h_G^\ast(5\epsilon)>0$, a contradiction. This completes the proof of the lemma.
\end{proof}

Given a continuous map $f\colon X\to X$, we say that a subset $S$ of $X$ is {\em pairwise recurrent} if for any $x,y\in S$ and $\gamma>0$, there is $l\ge1$ such that
\[
\max\{d(x,f^l(x)),d(y,f^l(y))\}\le\gamma.
\]

\begin{lem}
Let $f\colon X\to X$ be a continuous map and let $Y$ be a closed subset of $X$ such that
\begin{itemize}
\item $f(Y)\subset Y$,
\item $f$ has shadowing on $Y$,
\item $f|_{Y_b^f}$ is h-expansive for some $b>0$.
\end{itemize}
Let $g=f|_{Y_b^f}\colon Y_b^f\to Y_b^f$. Given $0<\beta<\epsilon$ and non-empty subsets $E,F$ of $X$, if
\begin{itemize}
\item $h_g^\ast(\epsilon)=0$,
\item $F\subset Y$,
\item $E$ is connected,
\item $E\subset\overline{F}$,
\item $F$ is pairwise recurrent, 
\item $\sup_{i\ge0}{\rm diam}\:f^i(F)\le\beta$,
\end{itemize}
then $E$ is a singleton.
\end{lem}

We use the following notation for the proof of Lemma 2.3.

\begin{notation}
\normalfont
For two sequences $A=(x_i)_{i=0}^k$, $B=(y_i)_{i=0}^l$ of points in $X$ with $x_k=y_0$, where $k,l\ge1$, we define $AB=(z_i)_{i=0}^{k+l}$ by
\begin{equation*}
z_i=
\begin{cases}
x_i&\text{if $0\le i\le k$}\\
y_{i-k}&\text{if $k\le i\le k+l$}
\end{cases}
.
\end{equation*}
\end{notation}

\begin{proof}
In order to obtain a contradiction, we assume the contrary, i.e., $E$ is not a singleton. Let
\[
0<2\gamma<\min\{{\rm diam}\:E,2b,\epsilon-\beta\}.
\]
We take $\delta>0$ such that every $2\delta$-pseudo orbit  $(w_i)_{i\ge0}$ of $f$ with $\{w_i\colon i\ge0\}\subset Y$ is $\gamma$-shadowed by some $p\in X$. Since $E$ is connected and $E\subset\overline{F}$, there are $x_0,x_1,\dots,x_k\in F$ such that
\begin{itemize}
\item $d(x_0,x_k)>2\gamma$,
\item $\sup_{0\le j\le k-1}d(x_j,x_{j+1})\le\delta$.
\end{itemize}
Since $F$ is pairwise recurrent, for every $0\le j\le k$, we have
\[
\max\{d(x_0,f^{l_j}(x_0)),d(x_j,f^{l_j}(x_j))\}\le\delta
\]
for some $l_j>0$. For $0\le j\le k$, let
\[
A_j=(p^{(j)}_i)_{i=0}^{l_j}=(x_0,f(x_0),\dots,f^{l_j-1}(x_0),x_0),
\]
a $\delta$-chain of $f|_Y$. For $0\le j\le k-1$, let
\[
B_j=(q^{(j)}_i)_{i=0}^{l_j}=(x_j,f(x_j),\dots,f^{l_j-1}(x_j),x_{j+1}),
\]
a $2\delta$-chain of $f|_Y$. Since $\sup_{i\ge0}{\rm diam}\:f^i(F)\le\beta$, we see that
\[
\sup_{0\le i\le l_j}d(p^{(j)}_i,q^{(j)}_i)\le\beta
\]
for all $0\le j\le k-1$. For $1\le j\le k$, let
\[
\tilde{B}_j=(r^{(j)}_i)_{i=0}^{l_j}=(x_j,f(x_j),\dots,f^{l_j-1}(x_j),x_{j-1}),
\]
a $2\delta$-chain of $f|_Y$. Since $\sup_{i\ge0}{\rm diam}\:f^i(F)\le\beta$, we see that
\[
\sup_{0\le i\le l_j}d(p^{(j)}_i,r^{(j)}_i)\le\beta
\]
for all $1\le j\le k$. Put
\[
L=l_0+l_k+2\sum_{j=1}^{k-1}l_j.
\]
Let
\[
(a_i)_{i=0}^{L}=A_0A_1\cdots A_{k-1}A_kA_{k-1}\cdots A_1
\]
and
\[
(b_i)_{i=0}^{L}=B_0B_1\cdots B_{k-1}\tilde{B}_k\tilde{B}_{k-1}\cdots\tilde{B}_1,
\]
$2\delta$-chains of $f|_Y$. Note that $a_0=b_0=a_L=b_L$ and
\[
2\gamma<\sup_{0\le i\le L-1}d(a_i,b_i)\le\beta.
\]
By Lemma 2.1, we obtain $h_G^\ast(\beta+2\gamma)>0$, where $G=f|_{Y_\gamma^f}\colon Y_\gamma^f\to Y_\gamma^f$. By $\gamma<b$ and $\beta+2\gamma<\epsilon$, we obtain
\[
h_g^\ast(\epsilon)\ge h_G^\ast(\epsilon)\ge h_G^\ast(\beta+2\gamma)>0,
\]
a contradiction. Thus, the lemma has been proved.
\end{proof}

The proof of the next lemma is left as an exercise for the reader. 

\begin{lem}
Let $D$ be a non-empty closed connected subset of $X$ which is not a singleton. Then, for any $\gamma>0$, there is a non-empty subset $E$ of $D$ such that
\begin{itemize}
\item $E$ is not a singleton,
\item $E$ is connected,
\item ${\rm diam}\:E\le\gamma$.
\end{itemize}
\end{lem}

By Lemmas 2.2, 2.3, and 2.4, we obtain the following lemma.

\begin{lem}
Let $f\colon X\to X$ be a continuous map and let $Y$ be a closed subset of $X$ such that
\begin{itemize}
\item $f(Y)\subset Y$,
\item $f$ has shadowing on $Y$,
\item $f|_{Y_b^f}$ is h-expansive for some $b>0$.
\end{itemize}
Given any non-empty closed subset $D$ of $Y$, if $D$ is connected and uniformly rigid, then $D$ is a singleton.
\end{lem}

\begin{proof}
Let $g=f|_{Y_b^f}\colon Y_b^f\to Y_b^f$. In order to obtain a contradiction, we assume the contrary, i.e., $D$ is not a singleton. Then, by Lemma 2.2 and Lemma 2.4, there are $0<\beta<\epsilon$ and a non-empty subset $E$ of $D$ such that
\begin{itemize}
\item $h_g^\ast(\epsilon)=0$,
\item $E$ is not a singleton,
\item $E$ is connected and uniformly rigid; and so pairwise recurrent,
\item $\sup_{i\ge0}{\rm diam}\:f^i(E)\le\beta$,
\end{itemize}
which contradicts Lemma 2.3; therefore, the lemma has been proved. 
\end{proof}

Finally, by Lemma 2.5, we prove Theorem 1.1.

\begin{proof}[Proof of Theorem 1.1]
Since $S$ is a non-empty uniformly rigid subset of $Y$, $\overline{S}$ is a non-empty closed uniformly rigid subset of $Y$. If $\overline{S}$ is not totally disconnected, then there is a connected component $D$ of $\overline{S}$ which is not a singleton. It follows that $D$ is a non-empty closed subset of $Y$ such that
\begin{itemize}
\item $D$ is not a singleton,
\item $D$ is connected and uniformly rigid,
\end{itemize}
which contradicts Lemma 2.5. This implies that $\overline{S}$ is totally disconnected and so zero-dimensional; therefore, we conclude that $S$ is zero-dimensional, completing the proof of the theorem.
\end{proof}

\section{A corollary of Lemma 2.3}

In this section, we consider closed pairwise recurrent invariant sets and present a corollary of Lemma 2.3. For a continuous map $f\colon X\to X$, we say that a subset $S$ of $X$ is {\em $f$-invariant} if $f(S)\subset S$. We say that a subset $S$ of $X$ is {\em pairwise recurrent} if for any $x,y\in S$ and $\gamma>0$, there is $l\ge1$ such that
\[
\max\{d(x,f^l(x)),d(y,f^l(y))\}\le\gamma.
\]
Note that if a closed $f$-invariant subset $S$ of $X$ is pairwise recurrent, then $f|_S\colon S\to S$ is a homeomorphism; and as a consequence of Proposition 1 of \cite{BHR}, $h_{\rm top}(f|_S)=0$.

A continuous map $f\colon X\to X$ is said to be {\em positively continuum-wise expansive} \cite{Kat} if there is $c>0$ such that for any non-empty closed connected subset $E$ of $X$ which is not a singleton, we have
\[
{\rm diam}\:f^i(E)>c
\]
for some $i\ge0$. By Theorem 5.8 of \cite{Kat}, we know that if a continuous map $f\colon X\to X$ is positively continuum-wise expansive and satisfies $h_{\rm top}(f)=0$, then $X$ is zero-dimensional.

Given a continuous map $f\colon X\to X$ and a closed $f$-invariant subset $S$ of $X$, we say that $f|_S\colon S\to S$ is {\em equicontinuous} if for any $\epsilon>0$, there is $\delta>0$ such that $d(x,y)\le\delta$ implies
\[
\sup_{i\ge0}d(f^i(x),f^i(y))\le\epsilon
\]
for all $x,y\in S$. It is known that if $f|_S$ is surjective and equicontinuous, then $f|_S$ is a homeomorphism and $f|_S^{-1}$ is also equicontinuous. If $f|_S$ is an equicontinuous homeomorphism, then $f|_S$ is {\em distal}, i.e.,
\[
\inf_{i\in\mathbb{Z}}d(f^i(x),f^i(y))>0
\]
for all distinct $x,y\in S$. It is also known that if $f|_S$ is a distal homeomorphism, then $S$ is pairwise recurrent. By results of \cite{AGW}, we know that whenever $S$ is zero-dimensional, the following conditions are equivalent:
 \begin{itemize}
\item $S$ is pairwise recurrent,
\item $f|_S$ is a distal homeomorphism,
\item $f|_S$ is an equicontinuous homeomorphism.
\end{itemize}

By Lemma 2.3 in Section 2, we obtain the following theorem.

\begin{thm}
Let $f\colon X\to X$ be a continuous map and let $S$ be a closed $f$-invariant subset of $X$ such that
\begin{itemize}
\item $f$ has shadowing on $S$,
\item $f|_{S_b^f}$ is h-expansive for some $b>0$.
\end{itemize}
If $S$ is pairwise recurrent (in particular, if $f|_S\colon S\to S$ is a distal homeomorphism), then
\begin{itemize}
\item $S$ is zero-dimensional,
\item $f|_S\colon S\to S$ is an equicontinuous homeomorphism. 
\end{itemize}
\end{thm}

\begin{proof}
Let $g=f|_{S_b^f}\colon S_b^f\to S_b^f$. Since $S$ is pairwise recurrent, we have $h_{\rm top}(f|_S)=0$. We take $0<\beta<\epsilon$ with
\[
h_g^\ast(\epsilon)=0.
\]
From Lemma 2.3, it follows that for any non-empty connected subset $E$ of $S$ which is not a singleton, we have
\[
{\rm diam}\:f^i(E)>\beta
\]
for some $i\ge0$. This implies that $f|_S$ is positively continuum-wise expansive. We conclude that $S$ is zero-dimensional, and thus $f|_S$ is an equicontinuous homeomorphism.   
\end{proof}

\end{document}